\documentclass{amsart}

\usepackage{amssymb, array, verbatim, amscd}
\usepackage{amsmath, amsfonts,amsthm}

\theoremstyle{plain}
\newtheorem{theorem}{Theorem}

\newtheorem{lemma}[theorem]{Lemma}

\theoremstyle{definition}

\newcommand{\CC}{{\mathcal C}}

\newcommand{\CF}{{\mathcal F}}

\newcommand{\CS}{{\mathcal S}}

\newcommand{\Hom}{\operatorname{Hom}}

\begin{document}

\title[Direct products and the contravariant hom-functor]{Direct products and the contravariant hom-functor}
\author{Simion Breaz}
\address{Babe\c s-Bolyai University, Faculty of Mathematics
and Computer Science, Str. Mihail Kog\u alniceanu 1, 400084
Cluj-Napoca, Romania}
\email{bodo@math.ubbcluj.ro}


\subjclass{16D90, 18A23, 18A30, 08A10, 08B25, 20K30}

\thanks{ }

\begin{abstract}
We prove in ZFC that if $G$ is a (right) $R$-module such that the
groups $\Hom_R(\prod_{i\in I}G_i,G)$ and $\prod_{i\in
I}\Hom_R(G_i,G)$ are naturally isomorphic for all families of
$R$-modules $(G_i)_{i\in I}$ then $G=0$. The result is valid even
we restrict to families such that $G_i\cong G$ for all $i\in I$.
 \end{abstract}

\keywords{direct product, hom-functor, pure-injective module}

\maketitle
\date{}

\section{Introduction}

It is well known that if $\CC$ is a category and $G$ is an object
in $\CC$ then the hom-functors $\Hom_\CC(-,G)$ and
$\Hom_\CC(G,-)$, from $\CC$ into the category of sets, are very
useful tools. For instance, the universal property of a direct
product can be described by the fact that all covariant
hom-functors $\Hom_\CC(G,-)$ preserve direct products, \cite[pp.70
and 117]{ML}: for any $G$ and any family $\CF=(G_i)_{i\in I}$ of
objects in $\CC$ such that there exists the (direct) product
$\prod_{i\in I}G_i$, there is natural bijection
$$\Phi_\CF:\Hom_\CC(G,\prod_{i\in I}G_i)\to \prod_{i\in
I}\Hom_\CC(G,G_i).$$  If $\CC$ has a null object (e.g. $\CC$ is
the category of pointed sets, of groups, of pointed spaces or a
category of modules) there are also canonical maps $u_i:G_i\to
\prod_{i\in I}G_i$ for all $i\in I$. The family
$\overline{u}_i=\Hom_R(u_i,G)$, $i\in I$, induces a natural
homomorphism $$\Psi_\CF:\Hom_\CC(\prod_{i\in I}G_i,G)\to
\prod_{i\in I}\Hom_\CC(G_i,G).$$ It is easy to see that in general
$\Psi_\CF$ is not a bijection (e.g. for pointed sets or for vector
spaces), and it is a natural question whether we can add
conditions on $G$ such that the induced maps $\Psi_\CF$ are
isomorphisms for all families $\CF$.

In the following we will prove (in ZFC) that for module categories
the only condition we can put is the trivial one, $G=0$. This
result is valid even if we restrict to all families $\CF$ which
consists in modules which are isomorphic copies of $G$. A similar
theorem was proved for abelian groups by Goldsmith and Kolman in
\cite[Theorem 3.3]{GK07} under an additional set theoretic
hypothesis (there exists a strongly compact cardinal), and our
result provides an answer to the questions from the end of
\cite{GK07}: Does ZFC suffice to prove that for every non-zero
abelian group $G$ there is a family $(G_i)_{i\in I}$ of abelian
groups (eventually with $G_i\cong G$ for all $i\in I$) such that
$\Hom(\prod_{i\in I}G_i,G)$ and $\prod_{i\in I}\Hom(G_i,G)$ are
not isomorphic?

The answer presented here is satisfactory since other commuting
properties of hom-functors are studied only in similar settings.
For instance the main characterizations for (self-)small modules
(see \cite[Proposition 1.1]{AM75}), using ascending chains of
submodules, are given for the hypothesis that the natural
homomorphism $\bigoplus_{i\in I}\Hom_R(G,G_i)\to
\Hom_R(G,\bigoplus_{i\in I} G_i)$ is an isomorphism, while slender
modules are characterized in a similar way in \cite[Corollary
III.1.5]{EM02}.

\section{Modules, $G$, such that $\Hom(-,G)$ preserves direct products}

Let $R$ be a unital ring, $G$ an $R$-module and $\CF=(G_i)_{i\in
I}$ a family of modules. As before, we denote by $u_i:G_i\to
\prod_{i\in I} G_i$ the canonical injections and by
$$\Psi_\CF:\Hom_R(\prod_{i\in I} G_i,G)\to \prod_{i\in I}
\Hom_R(G_i,G)$$ the natural homomorphism induced by the family
$\overline{u}_i=\Hom_R(u_i,G)$, $i\in I$. Following the
terminology used in \cite{GK07}, we say that $G$ is {\sl naturally
cosmall} if $\Psi_\CF$ is an isomorphism for all families $\CF$.
If $\Psi_\CF$ is an isomorphism for all families $\CF$ which
consist of isomorphic copies of $G$, then $G$ is called {\sl
naturally self-cosmall}. We will prove that the only naturally
self-cosmall module (hence the only naturally cosmall module) is
the trivial module $0$. In order to do this we start with a
technical lemma.

\begin{lemma}
Let $G$ be an $R$-module. If $\CF=(G_i)_{i\in I}$ is a family of
$R$-modules, $\varphi:\bigoplus_{i\in I}G_i\to \prod _{i\in I}G_i$
is the natural homomorphism and
$$\Upsilon_\CF:\Hom_R(\bigoplus_{i\in I}G_i, G)\to \prod_{i\in
I}\Hom_R(G_i,G)$$ is the natural isomorphism induced by the
canonical injections $v_i:G_i\to\bigoplus_{i\in I}G_i$, then
$$\Psi_\CF=\Upsilon_\CF\Hom_R(\varphi,G).$$
\end{lemma}

\begin{proof}
Let $\pi_i:\prod_{i\in I}\Hom_R(G_i,G)\to \Hom_R(G_i,G)$ be the
canonical projections. The standard proof, \cite[Theorem
43.1]{F70}, of the isomorphism $\Hom_R(\bigoplus_{i\in
I}G_i,G)\cong \prod_{i\in I}\Hom_R(G_i,G)$ shows that
$\Hom_R(v_i,G)=\pi_i\Upsilon_\CF$ for all $i\in I$.

Similarly, $\Hom(u_i,G)=\pi_i\Psi_{\CF}$. Now $\varphi v_i=u_i$,
so that, $\Hom(v_i,G)\Hom(\varphi,G)=\Hom(u_i,G)$, that is,
$\pi_i\Upsilon_\CF \Hom(\varphi,G)=\pi_i \Psi_\CF$ for all $i\in
I$. The conclusion follows immediately.
\end{proof}

\begin{theorem}
The following are equivalent for a module $G$: \begin{enumerate}
\item $G$ is naturally cosmall; \item $G$ is a naturally
self-cosmall module; \item $G=0$.
\end{enumerate}
\end{theorem}

\begin{proof} Only (2)$\Rightarrow$(3) requires a proof.
Let $G$ be a naturally self-cosmall module. For every cardinal
$\lambda$ we consider the canonical exact sequence $$0\to
G^{(\lambda)}\overset{\varphi}\to G^\lambda\to
G^\lambda/G^{(\lambda)}\to 0.$$ Applying the previous lemma, it
follows that $\Hom_R(\varphi,G)$ is an isomorphism, hence the
exactness of these sequences are preserved if we apply the functor
$\Hom_R(-,G)$. Then $G$ is pure-injective as a consequence of
\cite[1.2.19(e)]{GT06}.

Moreover, $\Hom_R(G^\lambda/G^{(\lambda)},G)=0$ for all $\lambda$
and we will prove that this is possible only if $G=0$.

Fix $\lambda$ an infinite ordinal. We observe that the set
$H=\{(g)_{\kappa<\lambda}\mid g\in G\}$ of constant functions
$$\lambda=\{\kappa\mid \kappa<\lambda\}\to G,\ \kappa\mapsto g,$$ is a
submodule in $G^\lambda$, and we claim that
$(H+G^{(\lambda)})/G^{(\lambda)}$ is pure in
$G^\lambda/G^{(\lambda)}$. In order to prove this claim we
consider a system of equations
$$(\CS)\ \ \sum_{j=1}^{k} a_{ij} X^j = (g^i)_{\kappa<\lambda}+G^{(\lambda)},
\ i = 1, \dots , n,\ a_{ij} \in
R,\ (g^i)_{\kappa<\lambda} \in H,\ n, k \in \mathbb{N},$$ 
in $G^\lambda/G^{(\lambda)}$. Suppose that the tuple
$$\left(\overline{x}^1=(x_{\kappa}^1)_{\kappa<\lambda}+G^{(\lambda)},\dots,
\overline{x}^k=(x_{\kappa}^k)_{\kappa<\lambda}+G^{(\lambda)}\right)$$
represents a solution for $(\CS)$ in $G^\lambda/G^{(\lambda)}$.
Then for every $i = 1, \dots , n$ the equalities $$\sum_{j=1}^{k}
a_{ij}x_\kappa^j = g^i$$ hold for almost all $\kappa<\lambda$. It
follows that  there is $\nu<\lambda$ such that the constant
functions
$$y^1=(x_{\nu}^1)_{\kappa<\lambda},\dots,y^k=(x_{\nu}^k)_{\kappa<\lambda}\in G^\lambda$$
satisfy the equalities $\sum_{j=1}^{k} a_{ij}y^j =
(g^i)_{\kappa<\lambda}$ for all $i = 1, \dots , n$. Then the tuple
$$\left(y^1+G^{(\lambda)},\dots,y^k+G^{(\lambda)}\right)$$ represents a solution
for $\CS$ in $(H+G^{(\lambda)})/G^{(\lambda)}$.

Therefore $(H+G^{(\lambda)})/G^{(\lambda)}\cong G$ is pure in
$G^\lambda/G^{(\lambda)}$. But $G$ is pure-injective, hence
$(H+G^{(\lambda)})/G^{(\lambda)}$ is a direct summand in
$G^\lambda/G^{(\lambda)}$.

If we suppose $G\neq 0$ we obtain
$\Hom_R(G^\lambda/G^{(\lambda)},G)\neq 0$, a contradiction.
\end{proof}

\subsection*{Acknowledgement} I would like to thank to Ciprian Modoi and Phill
Schultz for illuminating discussions on subjects related to the
main result of this note. I am thankful to the referee for her/his
suggestions who helped me to improve the presentation of the
paper.

 \end{document}